%% file: twocorner.tex
\documentclass[11pt]{article}

\usepackage{color}
\usepackage{graphicx}
\usepackage{amsmath}
\usepackage{amsxtra}
\usepackage{amssymb}
\usepackage{amsfonts}
\usepackage{amsthm}
\usepackage{amstext}
\usepackage{amsbsy}
\usepackage{amscd}
\usepackage[sans]{dsfont}

\theoremstyle{plain}
\newtheorem{theorem}{Theorem}

\newtheorem{lemma}{Lemma}

\newtheorem{proposition}[lemma]{Proposition}

\theoremstyle{definition}

\theoremstyle{remark}
\newtheorem{example}[lemma]{Example}

\def\defeq{:=}
\def\eqdef{=:}
\newcommand{\closure}[1]{\overline{#1}}
\newcommand{\dnconv}{\searrow}
\newcommand{\diam}{\operatorname{diam}}
\newcommand{\sign}{\operatorname{sign}}
\newcommand{\Leb}{\mathcal L}
\newcommand{\Leba}[1]{\Leb^{#1}}
\newcommand{\Lone}{\Leba1}
\newcommand{\loc}{{\operatorname{loc}}}
\newcommand{\isect}{\cap}
\newcommand{\Ctwo}{\Ck2}
\newcommand{\Czero}{\Ck0}
\newcommand{\qeq}{\quad\eqv\quad}
\newcommand{\eqv}{\Leftrightarrow}
\newcommand{\union}{\cup}
\newcommand{\set}[1]{\{#1\}}
\newcommand{\conv}{\rightarrow}
\newcommand{\topref}[2]{\overset{\text{\eqref{#1}}}{#2}}
\newcommand{\toprefb}[3]{\overset{\text{\eqref{#1}}}{\underset{\text{\eqref{#2}}}{#3}}}
\newcommand{\Cinf}{\Ck\infty}
\newcommand{\esssup}{\operatorname{esssup}}
\newcommand{\eps}{\epsilon}
\newcommand{\woba}[2]{\mathcal{W}^{#1,#2}} 
\newcommand{\setdiff}{\backslash}
\newcommand{\R}{\mathds{R}}
\newcommand{\bdry}{\partial}
\newcommand{\spC}{\mathcal{C}}
\newcommand{\Cone}{\Ck1}
\newcommand{\Ck}[1]{\spC^{#1}}
\newcommand{\pd}[1]{\partial_{#1}}
\newcommand{\cli}[2]{{[#1,#2]}}
\newcommand{\boi}[2]{{]#1,#2[}}
\newcommand{\loi}[2]{{]#1,#2]}}

\newcommand{\subeq}[2]{\mathord{\underbrace{\mathop{#1}}_{#2}}}
\newcommand{\supeq}[2]{\mathord{\overbrace{\mathop{#1}}^{#2}}}
\newcommand{\vlen}[1]{|#1|}
\newcommand{\nperp}{\nabla^\perp}
\newcommand{\trace}{\operatorname{tr}}
\newcommand{\half}{\frac12}
\newcommand{\Lap}{\Delta}
\newcommand{\const}{\text{const}}
\newcommand{\csep}{\quad,\quad}
\newcommand{\dotp}{\cdot}
\newcommand{\crossp}{\times}
\newcommand{\hess}{\nabla^2}
\newcommand{\tensor}{\otimes}
\newcommand{\pt}{\partial_t}
\newcommand{\ndiv}{\nabla\dotp}
\newcommand{\ncurl}{\nabla\crossp}

\newcommand{\pglau}{\beta} 
\newcommand{\vmssonic}{\vms_1}

\newcommand{\Machsmall}{\overline\Mach}

\newcommand{\rhsf}{F}
\newcommand{\sstf}{\overline\stf}
\newcommand{\scl}{\eps}

\newcommand{\densz}{\dens_{\max}}

\newcommand{\scriptn}{^{n}}

\newcommand{\scln}{\scl\scriptn}
\newcommand{\stfn}{\stf\scriptn}
\newcommand{\sstfn}{\sstf\scriptn}
\newcommand{\Machn}{\Mach\scriptn}
\newcommand{\vvn}{\vv\scriptn}
\newcommand{\vvin}{\vvi\scriptn}

\newcommand{\ssndn}{\ssnd\scriptn}
\newcommand{\densn}{\dens\scriptn}
\newcommand{\densin}{\densi\scriptn}

\newcommand{\sstflim}{\sstf}

\newcommand{\Pola}{\Theta}

\newcommand{\isenc}{\gamma}

\newcommand{\polalo}{\pola_0}
\newcommand{\polahi}{\pola_1}

\newcommand{\www}{w}
\newcommand{\wwi}{w_\infty}
\newcommand{\wpot}{\Phi}

\newcommand{\subso}{\underline\stf}

\newcommand{\varstf}{\tilde\stf}
\newcommand{\vxi}{v^x_\infty}

\newcommand{\Ns}{\set{\stf<0}}
\newcommand{\cDom}{\closure\Dom}

\newcommand{\Domi}{\Dom\union\set\infty}

\newcommand{\Dom}{\Omega}
\newcommand{\Slipb}{\bdry\Body}
\newcommand{\Body}{B}
\newcommand{\vx}{v^x}
\newcommand{\vy}{v^y}

\newcommand{\Gam}{\Gamma}
\newcommand{\Dt}{D_t}

\newcommand{\vort}{\omega}

\renewcommand{\Im}{\operatorname{Im}}

\newcommand{\dens}{\varrho}
\newcommand{\vpot}{\phi}
\newcommand{\cpot}{\Phi}
\newcommand{\stf}{\psi}

\newcommand{\piv}{p}
\newcommand{\pif}{\hat\piv}
\newcommand{\ipif}{\pif^{-1}}
\newcommand{\pp}{P}
\newcommand{\ppf}{\hat\pp}

\renewcommand{\vec}[1]{\mathbf{#1}}

\newcommand{\rad}{r}
\newcommand{\vn}{\vec n}
\newcommand{\vs}{\vec s}
\newcommand{\xx}{{\vec x}}
\newcommand{\xxmax}{\xx_+}
\newcommand{\xxmin}{\xx_-}

\newcommand{\pola}{\theta}
\newcommand{\ssnd}{c}
\newcommand{\Mach}{M}

\newcommand{\pr}{\pd\rad}
\newcommand{\po}{\pd\pola}
\newcommand{\px}{\pd x}
\newcommand{\py}{\pd y}

\newcommand{\Machi}{\Mach_\infty}

\newcommand{\vvi}{\vv_\infty}

\newcommand{\ssndi}{\ssnd_\infty}

\newcommand{\densi}{\dens_\infty}

\newcommand{\stfi}{\stf_\infty}

\newcommand{\gdiv}{\vec g}

\newcommand{\hdiv}{\hat\tau}

\newcommand{\vms}{\mu}

\newcommand{\defm}[1]{\emph{#1}}
\newcommand{\vv}{\vec v}

\begin{document}

\title{Subsonic irrotational inviscid flow around certain bodies with two protruding corners}%
\author{Volker Elling}
\date{Dedicated to Robert Finn on the occasion of his 95th birthday}
\maketitle
\begin{abstract}
  We prove non-existence of nontrivial uniformly subsonic inviscid irrotational flows around several classes of solid bodies 
  with two protruding corners, in particular vertical and angled flat plates; horizontal plates are the only case where solutions exists.
  This fills the gap between classical results on bodies with a single protruding corner on one hand and recent work on bodies with three or more protruding corners. 

  Thus even with zero viscosity and slip boundary conditions solids can generate vorticity, in the sense of having at least one rotational but no irrotational solutions. 
  Our observation complements the commonly accepted explanation of vorticity generation based on Prandtl's theory of viscous boundary layers. 
\end{abstract}

\section{Summary}

Consider steady planar flow around a solid body. 
The force exerted by the fluid on the body is a crucial quantity for aero- and hydrodynamics. 
Particularly important are the case of smooth boundaries and the case of bodies with a single corner which is \defm{protruding}, 
meaning the angle through the exterior is greater than $180^\circ$ so that the corner protrudes into the fluid, in contrast to the case of \defm{receding} corners.  
In aerodynamics such a body idealizes a cross-section of an aircraft wing, with the corner representing the trailing edge. 

The classical \defm{Kutta-Joukowsky theory} models this situation with an incompressible irrotational inviscid fluid and a slip condition on the solid boundary.
Given the velocity, there are infinitely many solutions, parametrized by \defm{circulation} $\Gam$, but only one of them satisfies the \defm{Kutta-Joukowsky condition}, 
namely that the velocity is bounded at the corner. Calculating the resulting pressure forces yields a well-known formula for lift, 
component of the force perpendicular to the velocity at infinity. The formula is in reasonable agreement\footnote{%
  see fig.\ 6.7.10 and surrounding text in \cite{batchelor}} with experimental data 
at least in some physical regimes.

The incompressible case requires looking for a harmonic \defm{stream function} satisfying a zero Dirichlet boundary condition;
the gradient rotated is the velocity. 
A natural Hilbert-space approach to such elliptic problems yields existence of solutions in spaces with locally square-integrable gradient. 
At receding corners such gradients are always bounded, but at protruding corners generally not, unless the problem data satisfies a single scalar real constraint\footnote{%
  see the discussion in \cite{elling-hyp2016} for details}. 
Mathematically this is the crucial difference between the two types of corners. 

The compressible subsonic case requires solving a \defm{nonlinear} equation, which is far more difficult. 
Frankl and Keldysh \cite{frankl-keldysh-1934} considered the low-Mach limit;
after seminal work of Morrey \cite{morrey-1938} on the foundations of 2d nonlinear elliptic PDE, 
Shiffman \cite{shiffman-exi-potf}, Bers \cite{bers-exi-uq-potf} and finally Finn and Gilbarg \cite{finn-gilbarg-uniqueness}
were able to give a rather complete subsonic generalization of Kutta-Joukowsky theory. 

In the incompressible case unbounded velocities are merely undesirable, 
but in the compressible case they are impossible,
at least for common pressure laws such as polytropic with isentropic coefficient greater than one. 
To quote Finn and Gilbarg \cite[p.\ 58]{finn-gilbarg-uniqueness}:
\begin{quote}
  ``Unlike the case of incompressible fluids, it appears very likely that in the theory of subsonic flows the Kutta-Joukowsky condition need not be imposed as an added hypothesis,
  but rather is a consequence of the subsonic character of the flow.'' 
\end{quote}
The fluid density reaches zero at a finite \defm{limit speed} and has no sensible definition at higher speeds:
fluid volumes moving to regions of ever lower pressure acquire only a finite velocity from acceleration by the pressure gradients.
A closely related observation:
fluid inside a piston expanding to near-vacuum can only perform a finite amount of mechanical work,
or conversely, it takes only a finite amount of energy to compress a large quasi-vacuum volume of fluid to any positive finite density. 
(These properties are not immediately obvious; they are calculated from the pressure law under consideration and may be false for ``exotic'' pressure laws.)

Bodies with \emph{several} protruding corners appear to have been studied mostly in the incompressible case, 
where conformal mapping techniques allow explicit solutions for a large variety of particular profiles;
less attention has been paid to unbounded velocities since they are tolerated by since incompressible models.
Starting from the prior work of Finn and Gilbarg and earlier authors,
\cite{elling-polylow} considers compressible flow around bounded simple polygons (which have at least three protruding corners), 
showing that there are \emph{no} nontrivial\footnote{A \defm{trivial} flow is $\vv=0$.} 
low-Mach number solutions, regardless of how the corner locations are chosen;
since their coordinates provide ample free parameters, this non-existence may be a bit surprising on the surface. 
The proof works by reduction to the incompressible case, where it is shown that no solution has bounded velocity at every corner.
This statement in turn is reduced to the \defm{utility graph} theorem: 
the bipartite graph $K_{3,3}$, with three vertices connected to each of three other vertices, 
does not have an embedding into the plane; 
the three corners represent one set of vertices while the other set corresponds to the region of positive streamfunction, the region of negative streamfunction
and the body, at which the streamfunction is zero. 
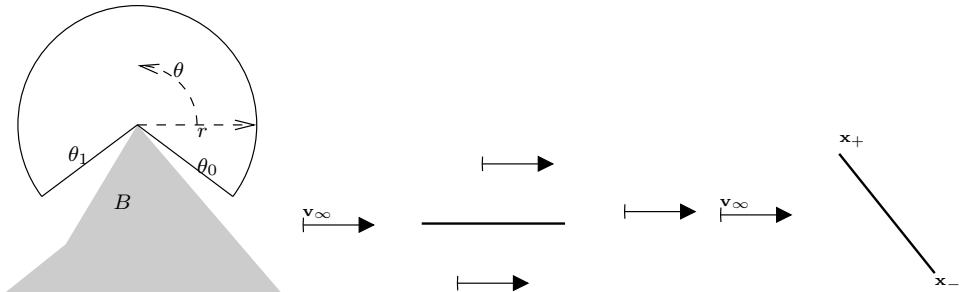
\begin{figure}
  \input{subsol.pstex_t}%
  \hfil%
  \input{horizontalplate.pstex_t}%
  \hfil%
  \input{diagonalplate.pstex_t}%
  \hfil%
  \caption{Left: a protruding corner in a solid (shaded). Center: horizontal plate; $\vv=\vvi$ is the trivial solution. Right: angled plate.}
  \label{fig:diagonalplate}
\end{figure}

But the proof idea falls well short for profiles with \emph{two} protruding corners, since planar embeddings of the $K_{2,3}$ graph are easy to find. 
Besides, nontrivial flows are easy to find, for example the horizontal plate (fig.\ \ref{fig:diagonalplate} center). 
However, such examples should be considered ``non-generic'': 
given the profile, there is only one free real scalar parameter (circulation), but there are two separate real scalar constraints (from velocity boundedness at each corner); 
for ``generic'' profiles the problem appears overdetermined. 

This heuristic argument is made precise in this article in several ways. 
We prove non-existence of nontrivial compressible uniformly subsonic flows for two (overlapping) classes:
profiles where the corners are the points of maximal and minimal vertical coordinate of the body (Theorem \ref{th:yminmax}),
and profiles that are symmetric across the flow axis, with corners not on the axis (Theorem \ref{th:verticalsymmetry}). 
These two theorems show non-existence for all non-horizontal flat plates, 
as well as many of the \defm{K\'arm\'an-Trefftz} symmetric lenses. 

Additional theorems covering more classes can be imagined, 
but it seems difficult to give a single criterion covering arbitrary profiles. 
For those we focus on the low-Mach case and prove (Theorem \ref{th:smallmach}):
if a profile does not admit nontrivial incompressible flows with bounded velocity, 
then nontrivial compressible flows with Mach numbers below a certain threshold do not exist either. 
The incompressible case is easier to check, for example by conformal mapping techniques;
indeed they allow us to cover all remaining lens cases.

If irrotational subsonic inviscid flows do not exist, there are several alternatives. 
One option is to consider transonic solutions, e.g.\ \defm{supersonic bubbles} at the solid, 
which are commonly observed at smooth outwardly curved boundary parts;
protruding corners can be considered infinite-curvature limits. 
Transition to supersonic is less likely at low Mach numbers; besides, supersonic bubbles generally end in shock waves 
that produce vorticity (mathematically rigorous existence proofs for such bubbles are still a subject of ongoing research). 
Another option is that we may not converge to a steady state as time passes to infinity (consider the von K\'arm\'an vortex strait). 

But the most obvious option is to consider that in reality vorticity is shed from solid surfaces; this is an important mechanism for generating \defm{drag}
in the low-viscosity low-Mach regime. 
Indeed \cite{berg-cpam-1962} has proven existence of subsonic flows around the vertical flat plate (to mention only one case), with vortex sheets emanating downstream from the endpoints.
Hence we have examples of obstacles that do not admit irrotational subsonic flows, but do admit rotational ones. 
In this sense, solids with more than one protruding corner can produce vorticity even with slip boundary conditions and without any viscosity. 

Previously the commonly accepted explanation of vorticity generation has been through the study of thin viscous boundary layers, by Prandtl \cite{prandtl-1922} and others. 
Prandtl's theory is certainly correct and provides more detailed insight; besides, it can explain vorticity generation and flow separation in the absence of corners. 
Simplicity is the main advantage of our alternate explanation: it does not require viscosity, non-slip boundary conditions, let alone details of boundary layers. 

\cite{prandtl-1922} and others already observed that infinite velocities at protruding corners are physically unreasonable. 
Our contribution is mathematical rigor: 
some exceptional shapes do allow irrotational flows with velocity bounded at every corner; 
that a particular body is not exceptional requires proof and that there are no irrotational subsonic flows around the body requires an additional proof.

\section{Background}

\subsection{Isentropic Euler}

The \defm{isentropic Euler} equations are
\begin{alignat*}{7} 
  0 &= \pt\dens + \ndiv(\dens\vv), \\
  0 &= \pt(\dens\vv) + \ndiv(\dens\vv\tensor\vv) + \nabla\pp ,
\end{alignat*} 
where $\vv$ is velocity while pressure $\pp=\ppf(\dens)$ is a strictly increasing function of density $\dens$. 
We only consider the \defm{polytropic} pressure law
\begin{alignat*}{7} \ppf(\dens) = \dens^\isenc \end{alignat*} 
with \defm{isentropic coefficient} $\isenc$ greater than $1$.
Assuming sufficient regularity the equations can be expanded into
\begin{alignat}{7} 
  0 &= \Dt\dens + \dens\ndiv\vv \csep \Dt = \pt + \vv\dotp\nabla,\label{eq:masst}\\
  0 &= \Dt\vv + \dens^{-1}\ppf_\dens(\dens)\nabla\dens.  \label{eq:vpre}\end{alignat} 
Linearizing the equations around 
\begin{alignat*}{7} \dens=\overline\dens=\const>0 \csep \vv=\overline\vv=\const=0 \end{alignat*}
yields 
\begin{alignat*}{7} 0 &= \pt\dens + \overline\dens \ndiv\vv, \\
0 &= 
\pt\vv + \overline\dens^{~-1} 
\ppf_\dens(\overline\dens) \nabla\dens, 
\end{alignat*} 
and subtracting $\ndiv$ of the lower equation from $\pt$ of the upper one yields
\begin{alignat*}{7} 
0 &= (\pt^2-\overline\ssnd^2\Lap)\dens, 
\end{alignat*}
which is the linear wave equation with \defm{sound speed} 
\begin{alignat*}{7} \overline\ssnd \defeq \sqrt{\ppf_\dens(\overline\dens)} \end{alignat*} 

\eqref{eq:vpre} can be rewritten
\begin{alignat}{7} 0 = \Dt\vv + \nabla\piv  \label{eq:v}\end{alignat} 
where the \defm{enthalpy per mass} $\piv=\pif(\dens)$ is defined (up to an additive constant) by
\begin{alignat}{7} \pif_\dens(\dens) = \dens^{-1}\ppf_\dens(\dens) \quad. \label{eq:pp-piv}\end{alignat}

\subsection{Potential flow}

Taking the curl of \eqref{eq:v} eliminates $\nabla\piv$, producing an equation for \defm{vorticity} $\vort=\ncurl\vv=v^y_x-v^x_y$:
\begin{alignat*}{7} 0 = \ncurl\pt\vv + \ncurl(\vv\dotp\nabla\vv) = ... = \Dt\vort + \vort\ndiv\vv \quad. \end{alignat*} 
Combined with \eqref{eq:masst} we obtain the transport equation
\begin{alignat*}{7} 0 = \Dt\frac\vort\dens\quad. \end{alignat*} 
If $\vort=0$ at $t=0$ (and $\dens>0$ throughout), then $\vort=0$ for all time.
There are important reasons to consider nonzero vorticity, 
as we pointed out in the introduction;
to demonstrate this, we explore consequences of assuming it is zero. 

$\ncurl\vv=0$ implies
\begin{alignat}{7} \vv = \nabla\vpot \label{eq:vpot}\end{alignat} 
for a scalar \defm{velocity potential} $\vpot$
(which is locally defined and may be multivalued when extended to non-simply connected domains). 

Henceforth we focus on stationary flow:
\begin{alignat}{7} 0 &= \ndiv(\dens\vv) \label{eq:mass} \quad, \\
0 &= \vv\dotp\nabla\vv+\nabla\piv \quad. \notag \end{alignat} 
Into the latter substitute \eqref{eq:vpot} to obtain\footnote{with $\vv^2=\vv\vv^T$ and $\hess$ the Hessian operator}
\begin{alignat*}{7} 0 = \hess\vpot\nabla\vpot + \nabla(\pif(\dens)) = \nabla\big( \half|\nabla\vpot|^2 + \pif(\dens) \big) \quad. \end{alignat*} 
This implies the \defm{Bernoulli relation}
\begin{alignat}{7} \half|\vv|^2+\pif(\dens) = \text{Bernoulli constant.}  \label{eq:ber}\end{alignat} 
$\pif_\dens(\dens)=\dens^{-1}\ppf_\dens(\dens)=\dens^{-1}\ssnd^2>0$, so $\pif$ is strictly increasing. 
Hence we can solve for
\begin{alignat}{7} \dens = \ipif \big( \text{Bernoulli constant} - \half|\vv|^2 \big) \label{eq:pividens}\end{alignat} 
for some maximal interval of $|\vv|$ closed at its left endpoint $0$; density $\dens$ reaches a maximum $\densz$ at $\vv=0$.

Substituting \eqref{eq:pividens} into \eqref{eq:mass} yields a second-order scalar differential equation for $\vpot$ called \defm{compressible potential flow}. 
After differentiation it is equivalent to\footnote{with Frobenius product $A:B=\trace(A^TB)$; note $A:\vec w^2=\vec w^TA\vec w$}
\begin{alignat}{7} 0 = \big(I-(\frac{\vv}{\ssnd})^2\big):\hess\vpot = \big(1-(\frac{v^x}{\ssnd})^2\big)\vpot_{xx} - 2\frac{v^x}{\ssnd}\frac{v^y}{\ssnd}\vpot_{xy} 
+ \big(1-(\frac{v^y}{\ssnd})^2\big)\vpot_{yy} \label{eq:comp-potf}\end{alignat} 
where $\ssnd$ is a function of $\dens$, hence of $\vv=\nabla\vpot$. 
The eigenvectors of the coefficient matrix $I-(\vv/\ssnd)^2$ are $\vv$ and\footnote{$\perp$ counterclockwise rotation by $\pi/2$} $\vv^\perp$, 
with eigenvalues $1-\Mach^2$ and $1$ where 
\begin{alignat*}{7} \Mach \defeq \vlen\vv/\ssnd \end{alignat*} 
is the \defm{Mach number}. 
Hence \eqref{eq:comp-potf} is elliptic in a given point if and only if 
\begin{alignat*}{7} \Mach < 1 \quad, \end{alignat*} 
i.e.\ if and only if velocity $|\vv|$ is below the speed of sound $\ssnd$; such flows are called \defm{subsonic}. 

A \defm{uniformly subsonic} flow has $\Mach\leq 1-\delta$ for some constant $\delta>0$ independent of $\xx$. 
Many classical results have been extended to the non-uniformly subsonic case, 
but in this article we prefer brevity over a slight improvement in generality.

\subsection{Streamfunction formulation}

We will need an alternative formulation of irrotational flow, which is obtained as follows: $\ndiv(\dens\vv)=0$
implies\footnote{with $\nabla^\perp=(-\partial_y,\partial_x)$}
\begin{alignat}{7} \dens\vv = -\nperp\stf \label{eq:rhovstf}\end{alignat} 
for a scalar \defm{stream function} $\stf$. $\vv\dotp\nabla\stf=0$ means that (except in \defm{stagnation points}, i.e.\ $\vv=0$)
the level sets of stream functions, called \defm{streamlines}, are integral curves of $\vv$, i.e.\ macroscopic trajectories of the fluid particles. 

Consider the Bernoulli relation \eqref{eq:ber} in the form
\begin{alignat}{7} \text{Bernoulli constant} = \subeq{\supeq{\half|\dens\vv|^2}{\vms}~\dens^{-2}+\pif(\dens) }{\eqdef \rhsf(\dens,\vms)} \quad \label{eq:berstf} \end{alignat} 
and apply the implicit function theorem. 
At solutions $(\dens,\vms)$ of \eqref{eq:berstf} 
that are vacuum-free and subsonic,
\begin{alignat*}{7} 
\frac{\partial\rhsf}{\partial\vms}
&=
\dens^{-2} > 0 \quad\text{and}
\\
\frac{\partial\rhsf}{\partial\dens}
&=
-\dens^{-3}|\dens\vv|^2+\pif_\dens(\dens)
=
\dens^{-1}(c^2-|\vv|^2) > 0 \quad,
\end{alignat*} 
so we obtain a solution
\begin{alignat}{1} \frac1\dens = \hdiv(\vms) \label{eq:hdiv}\end{alignat} 
for a strictly increasing function $\hdiv$ defined for $\vms$ in some maximal interval $\cli{0}{\vmssonic}$ 
for some constant $\vmssonic\in\loi0\infty$; 
for $|\vms|=\vmssonic$ the velocity is exactly sonic. 

Having solved the mass and Bernoulli equations it remains to ensure irrotationality\footnote{which is needed to recover the original velocity equation from the Bernoulli relation}:
\begin{alignat}{7} 0 = \ncurl \vv = \ncurl \frac{-\nperp\stf}{\dens} = -\ndiv\Big( \subeq{ \hdiv(\frac{|\nabla\stf|^2}{2}) \nabla\stf }{\eqdef\gdiv(\nabla\stf)} \Big) \label{eq:stf-divform} \end{alignat} 

Assuming sufficient additional regularity, differentiation yields after some calculation that
\begin{alignat}{7} 
0
&=
\big(1-(\frac{\vv}{\ssnd})^2\big):\hess\stf 
\notag\\&= 
\big(1-(\frac{v^x}{\ssnd})^2\big)\stf_{xx} - 2\frac{v^xv^y}{\ssnd^2}\stf_{xy} + \big(1-(\frac{v^y}{\ssnd})^2\big)\stf_{yy} 
\label{eq:stf-2d}
\end{alignat} 
which has the same coefficient matrix as \eqref{eq:comp-potf}; again it is elliptic if and only if the flow is subsonic. 

The incompressible limit of \eqref{eq:stf-2d} is obtained by (for example) considering sequences of solutions with velocities approaching $0$,
and hence (with fixed Bernoulli constant) sound speed converging to a positive constant. 
Correspondingly the ``Mach number'' of incompressible solutions is considered to be $0$.
In the limit we obtain
\begin{alignat*}{7} 0 &= -\Lap\stf \end{alignat*}
which is \eqref{eq:stf-divform} with $\hdiv=\const>0$. 
This can also be obtained (see e.g.\ \cite{klainerman-majda-singular}) along similar lines as for compressible flow from the unsteady incompressible Euler equations 
\begin{alignat*}{7} 0 &= \ndiv\vv, \\
0 &= \Dt\vv + \nabla\piv;
\end{alignat*} 
here $\dens=\const$, and $\piv$ is not a function of $\dens$ but rather a separate unknown making the second equation divergence-free.

Incompressible potential flows correspond to harmonic functions; 2d harmonic functions are conveniently represented by holomorphic functions of a single variable. 
To this end it is customary to consider the \defm{complex velocity}
\begin{alignat*}{5}
  \www &\defeq \vx-i\vy 
\end{alignat*}
as a function of
\begin{alignat*}{7}  
z &\defeq x+iy. 
\end{alignat*} 
Then 
\begin{alignat*}{5}
  \pd z^* \www = \half(\px+i\py)(\vx-i\vy)
  &= \half( \ndiv\vv - i\ncurl\vv ) .
\end{alignat*}
Hence $\www$ represents an \emph{incompressible} and \emph{irrotational} flow if and only if $\www$ is holomorphic. 

If so, it is convenient to use the \defm{complex velocity potential} $\wpot=\int^z\www~dz$; the lower endpoint of the integral is fixed (changing it only adds a constant);
the path does not matter locally since $\www$ is holomorphic, but for non-simply connected domains $\wpot$ may be multivalued
(the simplest example being the \defm{point vortex} $\wpot=\frac1{2\pi i}\log z$ which has multi-valued $\vpot$,
but corresponds to the single-valued $\vv=(2\pi)^{-1}|\xx|^{-2}(-y,x)$).
\begin{alignat*}{5}
  \wpot = \vpot + i\stf 
\end{alignat*}
is also holomorphic, satisfying Cauchy-Riemann equations
\begin{alignat*}{5}
  \vpot_x = \stf_y \csep \vpot_y = -\stf_x ,
\end{alignat*}
which yield
\begin{alignat*}{5}
  \vv &= \begin{bmatrix}
    \vx  \\ \vy
  \end{bmatrix} &&= \nabla\vpot = -\nperp\stf.
\end{alignat*}

\subsection{Slip condition}

At solid boundaries we use the standard \defm{slip condition}
\begin{alignat}{7} 0 = \vn \dotp \vv\quad, \label{eq:slip}\end{alignat} 
where $\vn$ is a normal to the solid. 
In the stream function formulation:
\begin{alignat*}{7} 0 = \vs \dotp \nabla\stf \quad, \end{alignat*} 
where $\vs$ is a tangent to the solid. 
In the latter case integration along connected components of (say) a piecewise $\Cone$ boundary yields 
\begin{alignat}{7} \Im\wpot = \stf = \const \label{eq:stf-const} \quad. \end{alignat} 
If the solid boundary has a single connected component, 
then we may add an arbitrary constant to $\stf$ without changing $\vv=-\nperp\stf$ to obtain the convenient zero Dirichlet condition
\begin{alignat*}{5}
  \Im\wpot = \stf = 0.
\end{alignat*}

\section{Assumptions, regularity and expansion at infinity}

We assume that $\Body$, the solid body, is a bounded closed set with boundary $\Slipb$ a continuous curve composed of finitely many
infinitely differentiable segments 
(for the sake of exposition; far less is needed);
$\Dom=\R^2\setdiff\Body$ is the fluid domain. 

We say $\xx\in\Slipb$ is a \defm{protruding corner} (fig.\ \ref{fig:diagonalplate} left) 
if there is a circular sector $S$ with radius $\eps>0$ centered in $\xx$
covering an angle $\delta>\pi$ so that the interior of $S$ is contained in $\Dom$. 

We consider $\stf$ that are infinitely differentiable in $\Dom$ and also at smooth points of $\Slipb$. 
This is no limitation: assume $\stf$ is in the most general possible class of stream functions, 
namely $\stf\in\woba1\infty(\Dom)$ (the space of functions with distributional derivatives that are essentially bounded functions), 
satisfying
\begin{alignat*}{7} \esssup_{\Dom} \half|\nabla\stf|^2 < \vmssonic \end{alignat*} 
(which is equivalent to $\esssup_{\Dom}\Mach < 1$),
and satisfying \eqref{eq:stf-divform} in the distributional sense, i.e.
\begin{alignat}{7} 
  0 = \int_{\Dom} \gdiv(\nabla\stf)\dotp\nabla\vartheta~d\xx 
  \label{eq:compweak}
\end{alignat} 
for every smooth function $\vartheta$ with compact support in $\Dom$, 
as well as the slip condition\footnote{for our domains $\woba1\infty$ has well-defined trace on the boundary} $\stf=0$ on $\Slipb$. 
In the compressible $\isenc>1$ case this is the largest reasonable class; 
as discussed in the introduction $\dens$ is not defined if $\vv$, or equivalently $\nabla\stf$, are unbounded.
Morrey estimates show that $\nabla\stf$ is $\Ck{0,\alpha}$ in $\Dom$, at $\Cinf$ parts of the boundary and also ``at infinity''
(after mapping it to $0$ by a change of coordinates $\vec a=\xx/|\xx|^2$). 
Then Schauder estimates and bootstrapping improve the regularity to our assumptions 
(see \cite{morrey-1938,bers-exi-uq-potf,finn-gilbarg-uniqueness}, 
\cite[Section 4]{elling-protrudingangle}, \cite[Chapter 6 and 12]{gilbarg-trudinger}),
in fact analyticity since $\ppf$ is analytic \cite{bernstein-1904,morrey-analyticity-i}.
(Although the classical work proves existence of such $\stf$, 
the statements of the uniqueness theorems do not clarify whether all essentially bounded $\vv$ are considered 
or merely those that are sufficiently smooth.)

In our work \cite{elling-polylow} on simple polygons, which have at least three protruding corners, we only needed 
an asymptotic expansion 
\begin{alignat*}{7} \vv = (\vxi,0) + o(1) .  \end{alignat*} 
However, it appears that 
the case of \emph{two} protruding corners ---- similar to the classical case of one --- 
requires a deeper expansion (see \cite[Section 4 and 5]{finn-gilbarg-uniqueness};
the expansion already appears in \cite{ludford-1951,bers-exi-uq-potf} and other works in varying degrees of detail and rigour):
\begin{alignat}{7} (\vx,\vy) = (\vxi,0) + \frac{\Gam}{2\pi} \frac{\pglau(-y,x)}{x^2+\pglau^2y^2} + O(|\xx|^{-1-\delta}) \label{eq:vv-expansion}\end{alignat} 
where $\delta>0$ is small, $\vvi=(\vxi,0)$ with $\vxi\geq0$ is the velocity at infinity, $\ssndi$ the corresponding sound speed, 
$\Machi=|\vvi|/\ssndi$ the Mach number and
\begin{alignat*}{7} \pglau = \sqrt{1-\Machi^2} \end{alignat*} 
the \defm{Prandtl-Glauert} factor. 
This expansion is obtained using $\Ck{0,\alpha}$ regularity at infinity 
which permits Schauder estimates for a first-order elliptic system
for $\nabla_\xx\stf$ in $\vec a$ coordinates. Interested readers can find the details in \cite{finn-gilbarg-uniqueness},
while others may wish to merely assume the expansion, or simply consider the incompressible case where 
taking $\Machi=0$ recovers the familiar Laurent expansion 
\begin{alignat*}{7} \www = \wwi + \frac{\Gam}{2\pi i} \frac1z + O(|z|^{-1-\delta}). \end{alignat*} 

We will need an expansion for $\stf$ which does not seem to be given explicitly in the classical papers: 
\begin{proposition}
  \begin{alignat}{7} \stf = \densi \big( \vxi y  - \frac{\Gam}{2\pi} \pglau \log\sqrt{x^2+\pglau^2y^2} \big) + \const + O(|\xx|^{-\delta}) 
    \label{eq:stf-asy}\end{alignat} 
  where $O$ has $O(|\xx|^{-1-\delta})$ gradient. 
\end{proposition}
\begin{proof}
  We expand the right-hand side of
  \begin{alignat*}{7} \nabla\stf \topref{eq:rhovstf}{=} \dens\vv^\perp \end{alignat*}
  as follows:
  \begin{alignat}{7} \dens &\topref{eq:pividens}{=} \ipif(\text{Bernoulli constant}-\half|\vv|^2) 
    \notag\intertext{(expand at $\vv=\vvi$)}
  &= \ipif(...) - (\ipif)'(...)\vvi\dotp(\vv-\vvi) + O(|\vv-\vvi|^2)
  \notag\\&\topref{eq:vv-expansion}{=} \densi - \densi \ssndi^{-2} (\vxi,0) \dotp \big( \frac{\Gam}{2\pi} \frac{\pglau(-y,x)}{x^2+\pglau^2y^2} + O(|\xx|^{-1-\delta}) \big) + O(|\xx|^{-2})
  \notag\\&\overset{\ssndi^{-2}(\vxi)^2=1-\pglau^2}{=} 
  \densi + \densi \pglau \frac{\Gam}{2\pi} \frac{(1-\pglau^2) y}{x^2+\pglau^2y^2} \frac1{\vxi} + O(|\xx|^{-1-\delta})
  \label{eq:rhoexp}
  \end{alignat} \begin{alignat*}{7} 
  \stf_x 
  =
  -\dens\vy
  &\toprefb{eq:vv-expansion}{eq:rhoexp}{=}
  \densi \pglau \frac{\Gam}{2\pi} \frac{-x}{x^2+\pglau^2y^2} + O(|\xx|^{-1-\delta})
   \end{alignat*} \begin{alignat*}{7} 
  \stf_y
  =
  \dens\vx
  &\toprefb{eq:vv-expansion}{eq:rhoexp}{=}
  \densi
  \Big(
  \vxi 
  + \pglau \frac{\Gam}{2\pi} \frac{ - \pglau^2 y }{x^2+\pglau^2y^2}  
  \Big)
  + O(|\xx|^{-1-\delta})
  \end{alignat*} 
  The right-hand sides without the $O$ parts are the gradient of 
  \begin{alignat*}{7} \densi \big( \vxi y  - \frac{\Gam}{2\pi} \pglau \log\sqrt{x^2+\pglau^2y^2} \big) \quad; \end{alignat*} 
  integrating the $O(|\xx|^{-1-\delta})$ term from some fixed point along some path toward infinity yields $O(|\xx|^{-\delta})$
  plus an integration constant.
\end{proof}

\section{Subsolution}

\newcommand{\axx}{a^{xx}}
\newcommand{\axy}{a^{xy}}
\newcommand{\ayy}{a^{yy}}
\newcommand{\maxr}{\overline r}
\newcommand{\subsocoeffeps}{\delta}
\newcommand{\subsoexpeps}{\eps}
\newcommand{\polamid}{\pola_+}

The Laplace equation admits solutions $\stf(\rad,\pola)=\rad^\alpha\sin(\alpha\pola)$ 
that satisfy a zero Dirichlet condition on rays at angles $\pola=0$ and $\Pola=\pi/\alpha$; 
protruding corners have $\Pola\in\boi{\pi}{2\pi}$ so that $\alpha=\pi/\Pola<1$
meaning $\nabla\stf\sim\rad^{\alpha-1}$ is unbounded. 
The Laplace operator can be used after showing that $\vv\conv 0$ at corners,
but that is not necessary because subsolutions can be obtained for general operators:
\begin{proposition}
  \label{prop:subso}%
  Consider a differential operator\\
  \begin{alignat*}{7} L = -A(\xx):\nabla^2 = -\axx(\xx)\pd x^2 - 2\axy(\xx)\pd x\pd y - \ayy(\xx)\pd y^2 \end{alignat*}
  that is uniformly elliptic on a sector $\set{\polalo<\pola<\polahi,\ 0<\rad<\maxr}$ 
  (with $\maxr>0$, $\polahi-\polalo\leq 2\pi$) in polar coordinates $(\rad,\pola)$.  
  If the sector angle $\polahi-\polalo$ is greater than $\pi$,  
  then there exists a subsolution $\subso$ which is $\Ctwo$ in the sector and continuous on its closure with
  \begin{alignat}{7} 
  L\subso &\leq 0 \quad\text{in the sector, and} \notag \\
  \subso &\leq 0 \quad\text{on the radii $\set{\pola=\pola_q,\ 0\leq\rad\leq\maxr}$ for $q=0,1$,} \label{eq:zeroonradii}
  \end{alignat}
  so that 
  \begin{alignat}{7} \subso &\geq \rad^{1-\eps} \quad\text{on some ray $\set{\pola=\polamid,\ 0\leq\rad\leq\maxr}$}  \label{eq:infgrad}\end{alignat} 
  for some constants $\eps\in\boi01$ and $\polamid\in\boi{\polalo}{\polahi}$.
\end{proposition}
\begin{proof}
  Ansatz:
  \begin{alignat*}{7} \subso(\rad,\pola) &= \rad^{1-\eps} u(\pola). \end{alignat*} 
  At $\pola=0$, $L\subso\leq 0$ is 
  \begin{alignat*}{7} 0 
  &\leq 
  \big( \axx\pr^2 + 2\axy\rad^{-1}(\pr-\rad^{-1})\po + \ayy(\rad^{-2}\po^2+\rad^{-1}\pr) \big) \subso
  \\&=
  \rad^{-1-\eps} \Big( \ayy ( u + u_{\pola\pola} ) -\eps \big( \axx (1-\eps) u  + 2\axy u_\pola + \ayy u \big) \Big)
  \end{alignat*} 
  and same at other $\pola$ if the $A$ coefficients are rotated accordingly. 
  To satisfy the inequality it is sufficient to solve $u+u_{\pola\pola}=1$ 
  and then take $\eps>0$ small, 
  using $|\axx|,|\axy|\leq C\ayy$ 
  for some constant $C<\infty$, by uniform ellipticity. 
  The solutions are
  \begin{alignat*}{7} u=1+a\cos(\pola-\polamid) \end{alignat*}
  where $a\geq 0$ is the amplitude and $\polamid$ the location of the maximum;
  the zeros on each side are at a distance $\arccos\frac{-1}{a}=\pi-\arccos\frac1a$ which, as $a$ ranges from $1$ to $\infty$, ranges from $\pi$ to arbitrarily close to but larger than $\pi/2$.
  Hence with $\polamid=\half(\polalo+\polahi)$ 
  we can satisfy \eqref{eq:zeroonradii} and \eqref{eq:infgrad}.
\end{proof}

\begin{proposition}
  \label{prop:singlesign}%
  A protruding corner $\xx$ does not have a neighbourhood $U$ so that $\stf\geq 0$ (or $\stf\leq 0$) in $U\isect\Dom$. 
  More generally, if some $\varstf\in\Ctwo(\Dom)\isect\woba1\infty(\Dom)$ solves $A(\xx):\nabla^2\varstf=0$ in $U\isect\Dom$ 
  for bounded and uniformly positive definite $A$, as well as $\varstf=0$ in the corner, 
  then $\varstf\geq 0$ on $U\isect\Dom$ means $\nabla\varstf$ is not bounded. 
\end{proposition}
\begin{proof}
  Assume otherwise. 
  We may choose $U$ to be a ball $\set{0<\rad<\maxr}$ (with $\maxr>0$ and polar coordinates $(\rad,\pola)$ centered in $\xx$)
  so that there is --- by definition of ``protruding corner'' --- a sector $\set{\polalo<\pola<\polahi,\ 0<\rad<\maxr}$
  with angle $\polahi-\polalo$ greater than $\pi$. 
  $\varstf\geq 0$ and $A:\nabla^2\varstf=0$ implies $\varstf>0$ in the interior of the sector, by the strong maximum principle. 
  We may shrink the sector slightly so that $\varstf>0$ on its closure,
  while keeping $\polahi-\polalo>\pi$ and $\maxr>0$.
  
  By Proposition \ref{prop:subso} we can pick a $\subso$ so that for any scalar $\iota>0$ we have 
  $-A:\nabla^2(\iota\subso)\leq 0\leq-A:\nabla^2\varstf$ in the sector and $\iota\subso\leq0\leq\varstf$ on its radii.
  Moreover $\varstf>0$ on the compact arc $\set{\polalo\leq\pola\leq\polahi,\ \rad=\maxr}$
  where $\subso$ and $\varstf$ are continuous, so by taking $\iota>0$ sufficiently small we have $\iota\subso\leq\varstf$ there.
  But then $\iota\subso\leq\varstf$ on the entire sector boundary, so by the comparison principle 
  $\iota\subso\leq\varstf$ in the entire sector, in particular $\varstf\geq \iota\rad^{1-\eps}$ on some ray, 
  which is not possible if $\nabla\varstf$ is bounded and $\varstf=0$ in the corner.
\end{proof}

For later use we recall the following classical uniqueness results (see \cite[section 6]{finn-gilbarg-uniqueness}).
They are also well-known for incompressible flow and the reader may wish to skip them in a first pass. 
\begin{proposition}
  \label{prop:vgam-uq}%
  Given $\vvi$ and $\Gam$ there is at most one compressible uniformly subsonic flow around the body.
\end{proposition}
\begin{proof}
  Assume there are two different ones, $\stf_q$ for $q=0,1$, 
  for same $\vvi$ and $\Gam$ so that their difference $d=\stf_1-\stf_0$ satisfies
  \begin{alignat*}{7} \nabla d \topref{eq:vv-expansion}{=} O(|\xx|^{-1-\delta}) \end{alignat*} 
  The difference of the weak formulations \eqref{eq:compweak} yields
  \begin{alignat*}{7} 0 &= \int_{\Dom} \nabla\vartheta \dotp \big( 
  \gdiv( \nabla\stf_1 )
  -  
  \gdiv( \nabla\stf_0 )
  \big) d\xx
  =
  \int_{\Dom} \nabla\vartheta \dotp \overline A(\xx) \nabla d~d\xx
  \end{alignat*} 
  where
  \begin{alignat*}{7} \overline A(\xx) 
  &= 
  \int_0^1 \gdiv'(\nabla\stf_q) dq 
  \end{alignat*}
  with $\stf_q=q\stf_1+(1-q)\stf_0$. 
  Mach number is a decreasing function of $|\nabla\stf|$, so the by convexity the flows defined by $\stf_q$ are subsonic uniformly in $\xx$ as well as $q\in\cli01$. 
  Therefore $\overline A$ is uniformly positive definite.
  
  Consider $\vartheta=d$ as test function:
  \begin{alignat*}{7} 0 = \int_{\Dom} \nabla d\dotp\overline A(\xx)\nabla d~d\xx \end{alignat*} 
  ($d$ as a test functions is admissible because $d\in\woba1\infty_\loc$, 
  which is sufficiently regular to approximate $\nabla d$ in $\Lone_\loc$ by smooth test function gradients;
  moreover the integrand is $O(|\xx|^{-2-2\delta})$, hence absolutely integrable,
  allowing approximation by compactly supported gradients.)
  Since $\overline A$ is uniformly positive definite, we immediately obtain $\nabla d=0$ and hence $d=0$ (since $d=0$ on the
  slip boundary). 
\end{proof}

\begin{proposition}
  \label{prop:uniqueness-corner}%
  If the body has at least one protruding corner, 
  then given $\vvi$ there is at most one compressible flow. 
\end{proposition}
\begin{proof}
  Assume there are two different $\stf_q$ ($q=0,1$) for the same $\vvi$. 
  By \eqref{eq:stf-asy} their difference behaves like
  \begin{alignat*}{7} d = c_1 \log(x^2+\pglau^2y^2) + c_0 + o(1) \quad\text{as $|\xx|\conv\infty$.} \end{alignat*} 
  If $c_1$ is nonzero, then $\log(x^2+\pglau^2y^2)>0$ near infinity means $\sign d=\sign c_1$ there; 
  now $d=0$ on the slip boundary 
  implies 
  by the strong maximum principle
  that $\sign d=\sign c_1$ throughout all of $\Dom$. 
  But that contradicts Proposition \ref{prop:singlesign}. 
  Hence $c_1=0$. 
  Next it can be shown that $c_0=0$, for the same reason. 
  But then $d\conv 0$ at infinity, so the strong maximum principle implies $d=0$. 
\end{proof}

\section{Nonexistence by symmetry}

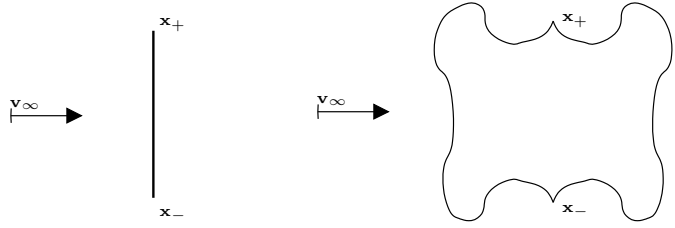
\begin{figure}
\hfil%
\input{verticalplate.pstex_t}
\hfil%
\input{recessedcorners.pstex_t}%
\hfil%
\caption{Left: flow onto a vertical plate. Right: flow onto a body symmetric across the flow axis, with $y$ extrema not attained in corners.}
\label{fig:recessedcorner}
\end{figure}

\begin{theorem}
  \label{th:verticalsymmetry}%
  Assume that $\Body$ is symmetric under $y\leftarrow-y$
  and has two protruding corners that are not on the horizontal axis. 
  Then there is no nontrivial uniformly subsonic flow, 
  and no nontrivial incompressible flow with bounded $\vv$. 
\end{theorem}
\begin{proof}
  Assume there is such a flow $\stf$.
  We may apply the symmetry transformation
  \begin{alignat*}{7} \stf(x,y) \leftarrow -\stf(x,-y), \end{alignat*} 
  which does not change $\Slipb$ and preserves $\stf=0$ there, whereas
  \begin{alignat*}{7} 
  \stf_x(x,y) &\leftarrow -\stf_x(x,-y), \\
  \stf_y(x,y) &\leftarrow \stf_y(x,-y), \\
  |\nabla\stf(x,y)| &\leftarrow |\nabla\stf(x,-y)|,
  \end{alignat*}
  and therefore, using $\dens=\ipif(|\nabla\stf|)$ and $(\vx,\vy)=\dens^{-1}(\stf_y,-\stf_x)$,
  \begin{alignat*}{7}
  \dens(x,y) &\leftarrow \dens(x,-y) , \\
  \vx(x,y) &\leftarrow \dens(x,y)^{-1}\stf_y(x,y) = \dens(x,-y)^{-1}\stf_y(x,-y) = \vx(x,-y), \\
  \vy(x,y) &\leftarrow -\dens(x,y)^{-1}\stf_x(x,y) = \dens(x,-y)^{-1}\stf_x(x,-y) = -\vy(x,-y).
  \end{alignat*}
  Hence $\vvi$ is unchanged, 
  so the new $\stf$ is another uniformly subsonic flow around the same body. 
  By Proposition \ref{prop:uniqueness-corner} there is at most one\footnote{The argument fails for symmetric bodies without corner; consider the incompressible flows around a circle, one for esch $\Gam$ and only the $\Gam=0$ one is symmetric.} such flow, 
  which is therefore symmetric under the transformation above,
  in particular $\stf=0$ on the horizontal axis.
  
  By symmetry there are protruding corners on both sides of the axis,
  so since by Proposition \ref{prop:singlesign} there are points with $\stf<0$ arbitrarily close to each protruding corner,
  such points occur on both sides of the axis. 
  But the entire horizontal axis has only points either inside the body or with $\stf=0$. 
  Hence $\Ns$ is disconnected, which we disprove as follows: 
  
  By \eqref{eq:stf-asy} $\stf_y=\densi\vxi+o(1)>0$ near infinity, and $\stf=0$ on the horizontal axis, 
  so the neighbourhood $\set{|\xx|\geq R}$ of infinity with $R>0$ sufficiently large has a half above the axis where $\stf>0$
  and a half below the axis where $\stf<0$.
  Hence there can be only one unbounded connected component of $\Ns$. 
  If $\Ns$ had a bounded connected component, then its closure would have $\stf=0$ on the boundary, but a negative value in its interior,
  so that a minimum is attained in the interior, contradicting the strong maximum principle for \eqref{eq:stf-2d}.
  Hence $\Ns$ is connected, a contradiction that completes the proof. 
\end{proof}

\begin{example}
    The vertical flat plate (fig. \ref{fig:recessedcorner} left) satisfies the conditions of Theorem \ref{th:verticalsymmetry}, 
    as does the profile with slightly sunken corners in fig. \ref{fig:recessedcorner} right. 
    The horizontal flat plate, for which flows with nonzero $\vxi$ exists, 
    satisfies all conditions except that the corners are on the axis. 
\end{example}

\section{Nonexistence for extremal $y$ corners}

Plates at other angles do not have the symmetry required in the previous section, but they can be discussed using another approach. 

\begin{theorem}
  \label{th:yminmax}
  Assume $\Body$ is not contained in a horizontal line. 
  Assume it has two protruding corners that are the lowest and highest point of $\Body$. 
  Then no nontrivial uniformly subsonic flow exists,
  nor nontrivial incompressible flow with bounded $\vv$. 
\end{theorem}
\begin{proof}
  Let
  \begin{alignat*}{7} \varstf = \stfi-\stf \csep \stfi = \densi\vxi y . \end{alignat*} 
  Then $\hess\stfi=0$, so $\varstf$ also satisfies
  \begin{alignat*}{7} 0 = (I-\ssnd^{-2}\vv^2):\hess\varstf . \end{alignat*} 
  By the strong maximum principle $\varstf$ cannot attain extrema in the set $\Dom$ of fluid points.

  $\stf=0$ on $\Slipb$ means $\varstf=\stfi=\densi\vxi y$ there, and $\densi\vxi>0$, 
  so since the protruding corners maximize resp.\ minimize $y$ over $\Slipb$, they do the same for $\varstf$. 

  At infinity, the expansion \eqref{eq:stf-asy} yields that 
  \begin{alignat*}{7} \varstf = c_1 \log(x^2+\beta^2y^2) + c_0 + o(1) \quad\text{as $|\xx|\conv\infty$} \end{alignat*} 
  where $c_1,c_0$ are some constants. 
  $\log(x^2+\beta^2y^2)$ is positive and dominant near infinity.
  If $c_1>0$, then $\varstf\conv+\infty$ at infinity,
  so $\varstf$ attains its minimum over $\cDom$ at the corner where $\varstf$ is smaller. 
  After subtracting the minimum value from $\varstf$ we have $\varstf=0$ in the corner, $\varstf\geq 0$ on $\Slipb$ 
  and hence a contradiction to Proposition \ref{prop:singlesign}.
  Similarly $c_1<0$ is excluded since it implies a global maximum of $\varstf$ in the upper corner. 
  Therefore $\varstf$ converges to $c_0$ at infinity. 
  But then we may repeat the same argument to obtain a contradiction.
\end{proof}

Since non-horizontal flat plates (fig.\ \ref{fig:diagonalplate} right) have two protruding corners 
in which the $y$ maxima and minima over the plate are attained, 
Theorem \ref{th:yminmax} immediately yields:
\begin{theorem}
  \label{th:flatplates}%
  There are no nontrivial uniformly subsonic flows around non-horizontal flat plates, 
  nor nontrivial incompressible flow with bounded $\vv$. 
\end{theorem}
This theorem shows the behaviour suggested in the introduction: 
for two protruding corners, a single parameter $\Gam$ has to satisfy two constraints;
that can be expected only in special cases, namely zero angle of attack.

\section{Nonexistence for low Mach numbers}

Theorem \ref{th:yminmax} covers all nonhorizontal plates but does not apply to the profile in Figure \ref{fig:recessedcorner} right,
which is covered by Theorem \ref{th:verticalsymmetry} which could only handle vertical plates. 
These theorems were proven by a combination of topological arguments and maximum principles. 
Of course many further arguments along these lines are conceivable,
but it seems that any number of them would cover some partially overlapping families of profiles while leaving other profiles uncovered.

For incompressible flow the non-existence of flows with bounded velocity is easier to decide 
due to an arsenal of complex analysis techniques. 
It is natural to extend these results to small-Mach compressible flows by linearizing the latter. 
Instead of a potentially cumbersome implicit function theorem approach
We give a rather short compactness argument based on Morrey estimates. 

\begin{theorem}
  \label{th:smallmach}%
  Assume there are no nontrivial incompressible flows with bounded $\vv$. 
  Then there exists a $\Machsmall>0$ so that there are no nontrivial 
  uniformly subsonic flows around $\Body$ with $\sup_{\Dom}\Machn\leq\Machsmall$. 
\end{theorem}
\begin{proof}
  Assume there is a sequence $(\stfn)$ of nontrivial compressible flows with $\sup_{\Dom}\Machn\dnconv 0$. 
  On one hand, at infinity $\densn\vvn$ converges to $\densin\vvin$, 
  by our definition a nonzero vector pointing exactly eastward. 
  On the other hand $\Body$ cannot be contained in a horizontal line
  (else $\vv=\const=\vvi\neq 0$ would be a nontrivial incompressible flow around it),
  so $\Slipb$ must (be smooth and) have non-horizontal tangent in some point,
  where the slip condition
  requires $\densn\vvn$ to 
  to be tangential, hence (zero\footnote{%
    We did not require corners to have exterior angle $<2\pi$,
    else we could simply obtain $\vvn=0$ in some point at the body.
  }
  or) $n$-uniformly \emph{not} eastward. 
  Therefore the diameter of the set of values $\set{\densn\vvn}$ is at least comparable to $|\densin\vvin|$, $n$-uniformly. 
  Hence with 
  \begin{alignat*}{7} \sstfn \defeq \stfn/\scln \csep \scln\defeq\diam\set{\nabla\stfn} \end{alignat*}
  we obtain that $\set{\nabla\sstfn}$ has diameter $1$ but is also $n$-uniformly bounded. 

  Since $\sstfn$ also satisfies the planar uniformly elliptic PDE 
  \begin{alignat*}{7} 0 = \big(I-(\frac{\vvn}{\ssndn})^2\big):\nabla^2\sstfn, \end{alignat*} 
  Morrey estimates
  (\cite{morrey-1938,elling-protrudingangle}, \cite[chapter 12]{gilbarg-trudinger})
  yield that $\nabla\sstfn$ is $n$-uniformly $\Ck\alpha$ locally uniformly in $\Domi$.
  $\Ck\alpha$ is compactly embedded in $\Czero$, 
  so we may restrict to a subsequence so that $(\sstfn)$ and $(\nabla\sstfn)$ converge in $\Czero$ locally in $\Domi$,
  and same for $\densn$ which is a smooth function of $\nabla\sstfn$, and from
  \begin{alignat*}{7} 0 = \ndiv(\frac1{\densn}\nabla\sstfn) \end{alignat*} 
  convergence in the distributional sense yields the incompressible limit
  \begin{alignat*}{7} 0 = \ndiv(\frac1{\densz}\nabla\sstflim). \end{alignat*} 
  with constant $\densz>0$.
  Moreover since $\nabla\sstfn$ is $n$-uniformly bounded, and since the boundary is piecewise $\Cinf$, 
  the slip condition $\sstflim=0$ is inherited from $\sstfn$.

  Hence $\sstflim$ is an incompressible flow around $\Body$,
  and also nontrivial since the limit inherits $\diam\set{\nabla\sstflim}=1$. 
  Contradiction.
\end{proof}

\newcommand{\isoh}{h}
\newcommand{\zetav}{\tilde z}
\newcommand{\circlepot}{\tilde\cpot}
\newcommand{\ktstf}{\cpot}
\newcommand{\ktexp}{\nu} 
\newcommand{\ktangle}{\beta}
\newcommand{\zfromzeta}{f}
\begin{figure}
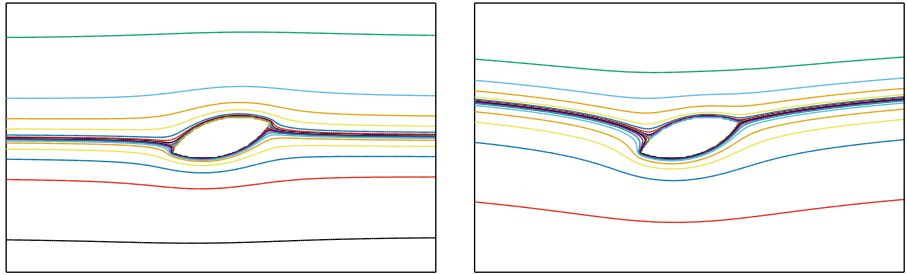

  \input{lens0.pstex}
  \input{lens.pstex}
  \caption{%
    Streamlines around K\'arm\'an-Trefftz lens with interior corner angle $270^\circ$, deflection $\ktangle=20^\circ$.
    Top left: $\Gam=0$.
    Top right: $\Gam$ chosen to yield bounded velocity at trailing (right) corner;
    rotate by $180^\circ$ to get the corresponding diagram for the leading corner.
  }
  \label{fig:lens}
  \end{figure}
\begin{example}
  As an application of Theorem \ref{th:smallmach} we consider symmetric lenses (fig.\ \ref{fig:lens}) which arise as a special case of K\'arm\'an-Trefftz profiles,
  which we recall here:
  
  The complex potential 
  \begin{alignat*}{7} \circlepot(\zetav) = \wwi\zetav + \wwi^*\frac1\zetav + \frac{\Gam}{2\pi i} \log\zetav \end{alignat*} 
  ($\log$ branch cut just below the negative real axis)
  satisfies the slip condition $\Im\circlepot=0$ on the unit circle.
  By rotational invariance we may conveniently rotate $\vvi$ instead of the body,
  considering $\wwi=|\wwi|\exp(i\ktangle)$, 
  where $\ktangle\in\loi{-\frac\pi2}{\frac\pi2}$ is the flow angle at infinity. 

  The idea of the K\'arm\'an-Trefftz transformation
  is to compose with a conformal map that deforms the circle into profiles with corners. 
  To create corners of prescribed angle, it is natural to use fractional powers $z\mapsto z^\ktexp$ (same branch cut). 
  They are conformal at $\pm 1$, but deform angles at $0,\infty$; we want the opposite, so we first apply the map
  \begin{alignat*}{7} \isoh(\zetav) \defeq \frac{\zetav+1}{\zetav-1} = \frac{1+1/\zetav}{1-1/\zetav}  \label{eq:isoh}\end{alignat*} 
  which swaps $\pm1$ with $0,\infty$. Define
  \begin{alignat*}{7} z = \zfromzeta(\zetav) \defeq \isoh(\isoh(\zetav)^\ktexp). \end{alignat*} 
  The image of the unit circle under $\zfromzeta$
  is a ``symmetric lens'' with corners $z=1=\zfromzeta(1)$ 
  and $z=-1=\zfromzeta(-1)$ of exterior (fluid) corner angle $\ktexp\pi$,
  with horizontal flat plate ($\ktexp=2$) and unit circle ($\ktexp=1$) as special cases.
  We consider $\ktexp\in\loi12$, i.e.\ protruding corners. 

  Since $\isoh$ is its own inverse, we get 
  \begin{alignat*}{7} \zfromzeta^{-1}(z) = \isoh(\isoh(z)^{1/\ktexp}) = \frac{ (z+1)^{1/\ktexp} + (z-1)^{1/\ktexp} }{ (z+1)^{1/\ktexp} - (z-1)^{1/\ktexp} }. \end{alignat*} 
  The $z$-plane complex potential is
  \begin{alignat*}{7} \cpot=\circlepot\circ\zfromzeta^{-1}. \end{alignat*}
  $1/\ktexp<1$ by choice, so $(\zfromzeta^{-1})'$ is unbounded at the corners $z=\pm1$. 
  Hence the complex velocity $\cpot'=(\circlepot'\circ\zfromzeta^{-1})\cdot(\zfromzeta^{-1})'$ can be bounded at the trailing edge $z=1$ only if 
  \begin{alignat*}{7} 0 = \circlepot'(1) 
  =
  \wwi - \wwi^* + \frac{\Gam}{2\pi i}
  \qeq
  \Gam = 4\pi \Im\wwi = 4\pi|\wwi|\sin\ktangle
  \end{alignat*} 
  Analogously requiring bounded velocity at the \defm{leading} corner $z=-1$ yields
  \begin{alignat*}{7} \Gam = -4\pi|\wwi|\sin\ktangle. \end{alignat*} 
  Clearly we cannot satisfy both conditions unless $\sin\ktangle=0$ (or $\wwi=0$), i.e.\ the case of horizontal lenses.

  We have constructed one flow for each $\Gam$, and by uniqueness (Proposition \ref{prop:vgam-uq}) there are no other ones. 
  Thus, K\'arm\'an-Trefftz symmetric lenses admit bounded-velocity nontrivial incompressible flows 
  only if the line between the corners is parallel to $\vvi$. 
  Otherwise, by Theorem \ref{th:smallmach} nontrivial low-Mach compressible flows do not exist, 
  even in the cases not covered by Theorem \ref{th:verticalsymmetry}
  and \ref{th:yminmax}.
  
  For horizontal lenses 
  we can exploit symmetry by putting a slip-condition straight wall onto the horizontal axis,
  obtaining an elliptic problem in the upper halfplane, 
  with all fluid-side corner angles smaller than $\pi$, 
  crucially yielding H\"older-continuous corner gradients in case of Dirichlet boundary conditions
  (see \cite{lieberman-pacj-1988,grisvard,mazya-plamenevskii} for unique solvability of the linearization 
  in weighted H\"older spaces and a priori estimates).
  However, some particular unsymmetric profiles should also admit proofs of existence; 
  that would require a more careful analysis of the corner constraints 
  and their relationship to circulation and angle of attack. 
\end{example}

\section*{Acknowledgement}

This material is based upon work partially supported by the
National Science Foundation under Grant No.\ NSF DMS-1054115
and by Taiwan MOST grant 105-2115-M-001-007-MY3.

\input{twocorner.bbl}

\end{document}

%% file: subsol.pstex_t
\begin{picture}(0,0)%
\includegraphics{subsol.pstex}%
\end{picture}%
\setlength{\unitlength}{3947sp}%
\begingroup\makeatletter\ifx\SetFigFont\undefined%
\gdef\SetFigFont#1#2#3#4#5{%
  \reset@font\fontsize{#1}{#2pt}%
  \fontfamily{#3}\fontseries{#4}\fontshape{#5}%
  \selectfont}%
\fi\endgroup%
\begin{picture}(1727,1810)(525,-1112)
\put(1726,-136){\makebox(0,0)[lb]{\smash{{\SetFigFont{8}{9.6}{\rmdefault}{\mddefault}{\updefault}{\color[rgb]{0,0,0}$r$}%
}}}}
\put(1576,239){\makebox(0,0)[lb]{\smash{{\SetFigFont{8}{9.6}{\rmdefault}{\mddefault}{\updefault}{\color[rgb]{0,0,0}$\pola$}%
}}}}
\put(1201,-586){\makebox(0,0)[lb]{\smash{{\SetFigFont{8}{9.6}{\rmdefault}{\mddefault}{\updefault}{\color[rgb]{0,0,0}$\Body$}%
}}}}
\put(1726,-361){\makebox(0,0)[lb]{\smash{{\SetFigFont{8}{9.6}{\rmdefault}{\mddefault}{\updefault}{\color[rgb]{0,0,0}$\polalo$}%
}}}}
\put(1051,-286){\makebox(0,0)[rb]{\smash{{\SetFigFont{8}{9.6}{\rmdefault}{\mddefault}{\updefault}{\color[rgb]{0,0,0}$\polahi$}%
}}}}
\end{picture}%

%% file: horizontalplate.pstex_t
\begin{picture}(0,0)%
\includegraphics{horizontalplate.pstex}%
\end{picture}%
\setlength{\unitlength}{3947sp}%
\begingroup\makeatletter\ifx\SetFigFont\undefined%
\gdef\SetFigFont#1#2#3#4#5{%
  \reset@font\fontsize{#1}{#2pt}%
  \fontfamily{#3}\fontseries{#4}\fontshape{#5}%
  \selectfont}%
\fi\endgroup%
\begin{picture}(2502,865)(-16514,-418)
\put(-16499, 62){\makebox(0,0)[lb]{\smash{{\SetFigFont{6}{7.2}{\rmdefault}{\mddefault}{\updefault}{\color[rgb]{0,0,0}$\vvi$}%
}}}}
\end{picture}%

%% file: diagonalplate.pstex_t
\begin{picture}(0,0)%
\includegraphics{diagonalplate.pstex}%
\end{picture}%
\setlength{\unitlength}{3947sp}%
\begingroup\makeatletter\ifx\SetFigFont\undefined%
\gdef\SetFigFont#1#2#3#4#5{%
  \reset@font\fontsize{#1}{#2pt}%
  \fontfamily{#3}\fontseries{#4}\fontshape{#5}%
  \selectfont}%
\fi\endgroup%
\begin{picture}(1387,1042)(-16514,-481)
\put(-16499, 62){\makebox(0,0)[lb]{\smash{{\SetFigFont{6}{7.2}{\rmdefault}{\mddefault}{\updefault}{\color[rgb]{0,0,0}$\vvi$}%
}}}}
\put(-15149,-436){\makebox(0,0)[lb]{\smash{{\SetFigFont{6}{7.2}{\rmdefault}{\mddefault}{\updefault}{\color[rgb]{0,0,0}$\xxmin$}%
}}}}
\put(-15749,464){\makebox(0,0)[lb]{\smash{{\SetFigFont{6}{7.2}{\rmdefault}{\mddefault}{\updefault}{\color[rgb]{0,0,0}$\xxmax$}%
}}}}
\end{picture}%

%% file: verticalplate.pstex_t
\begin{picture}(0,0)%
\includegraphics{verticalplate.pstex}%
\end{picture}%
\setlength{\unitlength}{3947sp}%
\begingroup\makeatletter\ifx\SetFigFont\undefined%
\gdef\SetFigFont#1#2#3#4#5{%
  \reset@font\fontsize{#1}{#2pt}%
  \fontfamily{#3}\fontseries{#4}\fontshape{#5}%
  \selectfont}%
\fi\endgroup%
\begin{picture}(970,1339)(-16514,-667)
\put(-16499, 62){\makebox(0,0)[lb]{\smash{{\SetFigFont{6}{7.2}{\rmdefault}{\mddefault}{\updefault}{\color[rgb]{0,0,0}$\vvi$}%
}}}}
\put(-15559,-622){\makebox(0,0)[lb]{\smash{{\SetFigFont{6}{7.2}{\rmdefault}{\mddefault}{\updefault}{\color[rgb]{0,0,0}$\xxmin$}%
}}}}
\put(-15559,575){\makebox(0,0)[lb]{\smash{{\SetFigFont{6}{7.2}{\rmdefault}{\mddefault}{\updefault}{\color[rgb]{0,0,0}$\xxmax$}%
}}}}
\end{picture}%

%% file: recessedcorners.pstex_t
\begin{picture}(0,0)%
\includegraphics{recessedcorners.pstex}%
\end{picture}%
\setlength{\unitlength}{3947sp}%
\begingroup\makeatletter\ifx\SetFigFont\undefined%
\gdef\SetFigFont#1#2#3#4#5{%
  \reset@font\fontsize{#1}{#2pt}%
  \fontfamily{#3}\fontseries{#4}\fontshape{#5}%
  \selectfont}%
\fi\endgroup%
\begin{picture}(2256,1393)(-17114,-692)
\put(-17099, 62){\makebox(0,0)[lb]{\smash{{\SetFigFont{6}{7.2}{\rmdefault}{\mddefault}{\updefault}{\color[rgb]{0,0,0}$\vvi$}%
}}}}
\put(-15559,-622){\makebox(0,0)[lb]{\smash{{\SetFigFont{6}{7.2}{\rmdefault}{\mddefault}{\updefault}{\color[rgb]{0,0,0}$\xxmin$}%
}}}}
\put(-15559,575){\makebox(0,0)[lb]{\smash{{\SetFigFont{6}{7.2}{\rmdefault}{\mddefault}{\updefault}{\color[rgb]{0,0,0}$\xxmax$}%
}}}}
\end{picture}%

%% file: twocorner.bbl
\providecommand{\bysame}{\leavevmode\hbox to3em{\hrulefill}\thinspace}
\providecommand{\MR}{\relax\ifhmode\unskip\space\fi MR }
\providecommand{\MRhref}[2]{%
  \href{http://www.ams.org/mathscinet-getitem?mr=#1}{#2}
}
\providecommand{\href}[2]{#2}